\documentclass[12pt]{amsart}

\usepackage{fullpage}
\usepackage[all]{xy}
\usepackage{latexsym,amsthm,amssymb,amsfonts,amsmath,mathrsfs}
\usepackage{multienum,color,tabls,listliketab,mathrsfs,extarrows,multicol,indentfirst,array}
\usepackage{fullpage, enumitem}
\usepackage{pifont,graphicx,txfonts}

\def\C{{\mathbb C}}

\def\J{{\mathcal J}}

\def\C{\mathbb{C}}
\def\G{\mathbb{G}}

\def\N{\mathbb{N}}

\def\TM{\mathbb{T}M}
\def\Z{\mathbb{Z}}

\def\Im{{\rm Im\/}}

\newtheorem{theorem}{Theorem}[section]
\newtheorem{lemma}[theorem]{Lemma}

\newtheorem{definition}[theorem]{Definition}

\newtheorem{proposition}[theorem]{Proposition}

\newenvironment{acknowledgement}[1][Acknowledgements]
{\begin{trivlist} \item[\hskip \labelsep {\bfseries #1}]}
{\end{trivlist}}

\newcommand{\Set}[2]{\left\{#1\bigg\vert #2\right\}}

\title{$E_1$-degeneration and $d'd''$-lemma}

\author{Tai-Wei Chen}
\address{Mathematics Department\\  National Tsing Hua University\\ Hsinchu, Taiwan}
\email{d937203@oz.nthu.edu.tw}

\author{Chung-I Ho}
\address{Mathematics Department\\  National Tsing Hua University\\  National Center of Theoretical Sciences\\ Mathematical Division\\ Hsinchu, Taiwan}
\email{ciho@math.cts.nthu.edu.tw}

\author{Jyh-Haur Teh*}
\address{Mathematics Department\\  National Tsing Hua University\\ Hsinchu, Taiwan}
\email{jyhhaur@math.nthu.edu.tw}
\date{}
\begin{document}

\begin{abstract}
For a double complex $(A, d', d'')$, we show that
if it satisfies the $d'd''$-lemma and the spectral sequence $\{E^{p, q}_r\}$ induced by $A$ does not degenerate at $E_0$, then it degenerates at $E_1$. We apply this result to prove the degeneration at $E_1$ of a Hodge-de Rham spectral sequence on compact bi-generalized Hermitian manifolds that satisfy a version of $d'd''$-lemma.
\end{abstract}

\let\thefootnote\relax\footnotetext{*Corresponding author.}

\maketitle

\section{Introduction}
Complex manifolds that satisfy the $\partial\overline{\partial}$-lemma enjoy some nice properties such as they are formal manifolds(\cite{DGMS}), their Bott-Chern cohomology, Aeppli cohomology
and Dolbeault cohomology are all isomorphic. Compact K\"ahler manifolds are examples of such manifolds. The Hodge-de Rham spectral sequence $E_*^{*,*}$ of a complex manifold $M$ is built from the double complex $(\Omega^{*, *}(M),\partial,\bar{\partial})$ of complex differential forms which relates the Dolbeault cohomology of $M$ to the de Rham cohomology of $M$. It is well known that $E^{p,q}_1$ is isomorphic to $H^p(M,\Omega^q)$ and the spectral sequence $E^{*,*}_r$ converges to $H^*(M,\C)$.
The goal of this paper is to prove an algebraic version of the result that the $\partial\overline{\partial}$-lemma implies the $E_1$-degeneration of a Hodge-de Rham spectral sequence.
The following is our main result.
\begin{theorem}\label{main result}
If a double complex $(A, d', d'')$ satisfies the $d'd''$-lemma and the spectral sequence $\{E^{p, q}_r\}$ induced by $A$ does not degenerate at $E_0$, then it degenerates
at $E_1$.
\end{theorem}

We define a spectral sequence that is analogous to the Hodge-de Rham spectral sequence of complex manifolds for bi-generalized Hermitian manifolds. Applying result above, we are able to show that for compact bi-generalized Hermitian manifolds that satisfy a version of $\partial\overline{\partial}$-lemma, the sequence degenerates at $E_1$.

\begin{acknowledgement}
The authors thank the referee for his/her extremely careful review which largely improves this paper.
\end{acknowledgement}

\section{Degeneration of a Hodge-de Rham spectral sequence}\label{ss1}

\begin{definition}
A spectral sequence is a sequence of differential bi-graded modules $\{(E^{*,*}_r, d_r)\}$ such that $d_r$ is of degree $(r,1-r)$ and $E^{p,q}_{r+1}$ is isomorphic to $H^{p,q}(E^{*,*}_r,d_r)$.
\end{definition}

\begin{definition}
A filtered differential graded module is a $\N$-graded module $A=\oplus_{k=0}^\infty A^k$, endowed with a filtration $F$ and a linear map $d:A\rightarrow A$ satisfying
\begin{enumerate}
\item $d$ is of degree 1: $d(A^k)\subset A^{k+1}$;
\item $d\circ d=0$;
\item the filtered structure is descending:
$$A=F^0A\supseteq F^1A\supseteq\cdots \supseteq F^kA\supseteq F^{k+1}A\supseteq\cdots ;$$
\item the map $d$ preserves the filtered structure: $d (F^kA)\subset F^kA$ for all $k$.
\end{enumerate}

\end{definition}

For $p, q, r\in \Z$, let

$
\begin{array}{cccccc}
   Z^{p, q}_r &= &\Set{\xi\in F^pA^{p+q}}{d \xi \in F^{p+r}A^{p+q+1}}, & Z^{p, q}_\infty & = & F^pA^{p+q}\cap \ker d \\
   B^{p, q}_r &= &F^pA^{p+q}\cap d F^{p-r}A^{p+q-1},  & B^{p, q}_\infty & = & F^pA^{p+q}\cap \Im d\\
  E^{p, q}_r  &= &\frac{Z^{p, q}_r}{Z^{p+1,q-1}_{r-1}+B^{p,q}_{r-1}}, & E^{p, q}_{\infty} & = & \frac{F^pA^{p+q}\cap \ker d}{F^{p+1}A^{p+q}\cap \ker d +F^pA^{p+q}\cap \Im d}
\end{array}$

 with the convention $F^{-k}A^{p+q}=A^{p+q}$ and $A^{-k}=\{0\}$ for $k\geq 0$. Let $d_r: E^{p, q}_r \rightarrow E^{p+r, q-r+1}_r$ be
  the differential induced by  $d: Z^{p, q}_r \rightarrow Z^{p+r, q-r+1}_r$.

Throughout this paper, we always assume that $A=\oplus_{p, q\geq 0} A^{p,q}$ is a double complex of vector spaces over some field
with two maps $d'_{p, q}:A^{p, q}\rightarrow A^{p+1, q}$ and $d''_{p, q}: A^{p,q}\rightarrow A^{p, q+1}$ satisfying
$d'_{p+1, q}d'_{p, q}=0, d''_{p, q+1}d''_{p, q}=0$ and $d'_{p, q+1}d''_{p, q}+d''_{p+1, q}d'_{p, q}=0$ for all $p, q\geq 0$. To make notation cleaner, we allow $p, q$ to be any integers
by defining $A^{p, q}={0}$ for $p<0$ or $q<0$.

Let $A^k=\bigoplus_{p+q=k}A^{p, q}$. Define
$$F^pA^k=\bigoplus_{s=p}^kA^{s, k-s}$$
For $p>k$, define $F^pA^k=\{0\}$.
This gives a descending filtration on $A^k$.

Let $d=d'+d''$.
The double complex $(A, d', d'')$ then defines a filtered differential graded module
$(A, d,F)$. Let $\{E_r^{p,q}\}$ be the corresponding spectral sequence. We are interested in the convergence of $E_r^{p,q}$.

\begin{definition}
Let $\{E^{p,q}_r\}$ be the spectral sequence associated to the double complex $(A, d', d'')$.
If
$d_s=0$ for all $s\geq r$, then we say that $\{E^{p,q}_r\}$ or $A$ degenerates at $E_r$.
\end{definition}

The following simple lemmas will be used frequently.

\begin{lemma}\label{vector2}
If $G'$ is a vector space and $H<G ,H<H'$ are subspaces of $G'$, the natural map
$\varphi: \dfrac{G}{H}\rightarrow \dfrac{G'}{H'}$
is injective if and only if $G\cap H'=H$, and is surjective if and only if $G'=G+H'$.
\end{lemma}

\begin{lemma}\label{injective and surjective}
Let $p, q, r\in \Z$.
There are inclusions
$$\cdots \subset B^{p, q}_0\subset B^{p, q}_1\subset \cdots\subset B^{p, q}_\infty\subset Z^{p, q}_\infty\subset \cdots \subset Z^{p, q}_1\subset Z^{p, q}_0\subset \cdots,$$
$$Z^{p+1,q-1}_{r-1}\subset Z^{p,q}_r \ \ , \ \  B^{p+1,q-1}_{r+1}\subset Z^{p,q}_r, \ \ d(Z^{p-r, q+r-1}_r)=B^{p, q}_r$$
\end{lemma}

\begin{definition}
Let $\alpha_{p, q, r}:E^{p, q}_{r+1} \rightarrow \frac{Z^{p, q}_r}{Z^{p+1, q-1}_{r-1}+B^{p, q}_r}$ be the map induced by the composition of inclusion and projection, and $\beta_{p, q, r}:E^{p, q}_r \rightarrow \frac{Z^{p, q}_r}{Z^{p+1, q-1}_{r-1}+B^{p, q}_r}$ be the map induced by the projection.
\end{definition}

\begin{proposition}\label{alpha and beta} Let $r\in \Z$. Then
\begin{enumerate}
\item $d_r=0$ if and only if $\beta_{p, q, r}$ is an isomorphism for all $p, q\in \Z$.

\item $d_r=0$ implies that $\alpha_{p, q, r}$ is an isomorphism for all $p, q\in \Z$.
\end{enumerate}
\end{proposition}

\begin{proof}
\begin{enumerate}
\item
We first note that the map $\beta_{p, q, r}$ is always surjective. By Lemma \ref{vector2}, $\beta_{p, q, r}$ is an isomorphism if and only if $Z^{p, q}_r\cap (Z^{p+1, q-1}_{r-1}+B^{p, q}_r)=Z^{p+1, q-1}_{r-1}+B^{p, q}_{r-1}$, or equivalently,
$B^{p, q}_r\subseteq Z^{p+1, q-1}_{r-1}+B^{p, q}_{r-1}$. The map
$d^{p-r, q+r-1}_r:E^{p-r, q+r-1}_r  \rightarrow E^{p, q}_r $ is the zero map
if and only if $\mbox{Im}d^{p-r, q+r-1}_r=\{0\}$. This is equivalent to
$d(Z^{p-r, q+r-1}_r)=B^{p, q}_r\subseteq Z^{p+1, q-1}_{r-1}+B^{p, q}_{r-1}$, which is equivalent
to $\beta_{p, q, r}$ being an isomorphism.

\item
We recall that the isomorphism $E^{p, q}_{r+1} \overset{\cong}{\longrightarrow} H^{p, q}(E^{*, *}_r, d_r)$ (see \cite[Proof of Theorem 2.6]{M}) is induced from some canonical projections and inclusions. If $d_r=0$, $H^{p, q}(E^{*, *}_r, d_r)\cong E^{p, q}_r$ and we have a commutative diagram
$$\xymatrix{ E^{p, q}_{r+1}  \ar[rr]^{\cong} \ar[rd]_{\alpha_{p, q, r}} & &  E^{p, q}_r \ar[ld]^{\beta_{p, q, r}}\\
&  \frac{Z^{p, q}_r}{Z^{p+1,q-1}_{r-1}+B^{p,q}_{r}} &}$$
By (1), $\beta_{p, q, r}$ is an isomorphism and hence $\alpha_{p, q, r}$ is an isomorphism.
\end{enumerate}
\end{proof}

\begin{definition}
Fix a pair of integers $(p, q)$.
For nonzero $\xi=\sum_i \xi_i\in \bigoplus_{i\geq 0}A^{p+i, q-i}$ where $\xi_i\in A^{p+i,q-i}$,
let $i_0=\min_i\{\xi_i\neq 0\}$.
We call $\xi_{i_0}$ the leading term of $\xi$ and denote it as $\ell^{p, q}(\xi)$. We define $\ell^{p, q}(0)=0$.
For $r\geq 1, p, q\in \Z$, let
$$\mathcal{E}^{p,q}_r:=\Set{\xi=\xi_0+\xi_1+\cdots +\xi_{r-1}}{\xi_i\in A^{p+i,q-i}, d\xi=d'\xi_{r-1}\notin \Im d'', \\
 \ell^{p, q}(\eta)\neq\xi_0 \hbox{ for all $d$-closed $\eta$}}$$
and
$$\mathcal{E}^{p, q-1}_0:=B^{p, q}_0-(Z^{p+1, q-1}_{-1}+B^{p, q}_{-1})$$
\end{definition}

\begin{lemma}\label{alpha}
Fix $r_0\geq 1$.
\begin{enumerate}
\item
If the map $\alpha_{p, q, r}$ is an isomorphism for all $p, q\in \Z, r\geq r_0$, then $\mathcal{E}^{p, q}_r=\emptyset$ for all $p, q\in \Z$, $r\geq r_0$.

\item
If the map $\alpha_{p, q, r_0}$ is not an isomorphism, then $\mathcal{E}^{p, q}_{r_0}\neq \emptyset$.
\end{enumerate}
\end{lemma}

\begin{proof}
Note that by Lemma \ref{vector2}, the surjectivity of $\alpha_{p, q, r}$ is equivalent to the condition
$$Z^{p, q}_r=Z^{p, q}_{r+1}+Z^{p+1,q-1}_{r-1}+B^{p,q}_{r}=Z^{p, q}_{r+1}+Z^{p+1,q-1}_{r-1}.$$

\begin{enumerate}
\item
Suppose that $\alpha_{p, q, r}$ is an isomorphism for all $r\geq r_0$. Then $Z^{p, q}_i=Z^{p, q}_{i+1}+Z^{p+1, q-1}_{i-1}$ for all $i\geq r_0$.
Assume that $\mathcal{E}^{p, q}_r\neq \emptyset$ for some $r\geq r_0, p, q\in \Z$. Let $\xi\in \mathcal{E}^{p, q}_r$.
By definition, $Z^{p, q}_{q+2}=Z^{p, q}_{q+3}=\cdots=Z^{p, q}_{\infty}$. So we may take $j>r$ such that $Z^{p, q}_j=Z^{p, q}_{\infty}$. Note that $\xi\in Z^{p, q}_r$.
Using the relation above, we may
write $\xi=\eta_1+\eta_2$ where $\eta_1\in Z^{p, q}_j, \eta_2\in Z^{p+1, q-1}_{j-2}+\cdots +Z^{p+1, q-1}_{r-1}$.
Since $\ell^{p, q}(\xi)\neq 0$, by comparing the degrees of both sides of $\xi=\eta_1+\eta_2$, we have
$\ell^{p, q}(\xi)=\ell^{p, q}(\eta_1)$. But $d\eta_1=0$ which contradicts to the fact that $\ell^{p, q}(\xi)$ is not the leading term of any $d$-closed element.

\item
Fix $r\geq 1$. Suppose that $\alpha_{p, q, r}$ is not an isomorphism, then
$Z^{p, q}_{r+1}+Z^{p+1,q-1}_{r-1}\subsetneqq Z^{p, q}_r$.

Let
$$\xi=\xi_0+\xi_1+\cdots +\xi_k \in Z^{p, q}_r- (Z^{p, q}_{r+1}+Z^{p+1,q-1}_{r-1})\mbox{ where }\xi_i\in A^{p+i,q-i}.$$
If $k>r-1$, let $\xi'=\xi_r+\xi_{r+1}+\cdots+\xi_k\in F^{p+r}A^{p+q}\subset F^{p+1}A^{p+q}$. We have
$$d\xi'=d\xi_r+\cdots +d\xi_k\in F^{p+r}A^{p+q+1}=F^{(p+1)+(r-1)}A^{(p+1)+(q-1)+1}$$
which means that $\xi'\in Z^{p+1, q-1}_{r-1}$. Let $\xi''=\xi-\xi'$. If $\xi''\in Z^{p, q}_{r+1}+Z^{p+1,q-1}_{r-1}$, then
$\xi=\xi'+\xi''\in Z^{p, q}_{r+1}+Z^{p+1,q-1}_{r-1}$ which contradicts to our assumption. Therefore $\xi''=\xi_0+\cdots +\xi_{r-1}\in Z^{p, q}_r-(Z^{p, q}_{r+1}+Z^{p+1,q-1}_{r-1})$.
Hence we may assume $\xi=\xi_0+\cdots +\xi_{r-1}$.

\begin{enumerate}
\item[(i)] Since $\xi\in Z^{p, q}_r$, by definition, $d\xi\in F^{p+r}A^{p+q+1}$.
But $d(\xi_0+\cdots+\xi_{r-2})+d''\xi_{r-1}\in A^{p, q+1}\oplus A^{p+1, q}\oplus \cdots \oplus A^{p+r-1, q-r+2}$. This forces $d(\xi_0+\cdots +\xi_{r-2})+d''\xi_{r-1}=0$ and hence $d\xi=d'\xi_{r-1}$.

\item[(ii)] If $d'\xi_{r-1}=d''\eta_r$ for some $\eta_r\in A^{p+r, q-r}$, then $d(\xi-\eta_r)=d'\xi_{r-1}-d'\eta_r-d''\eta_r=-d'\eta_r\in A^{p+r+1, q-r}\subset F^{p+(r+1)}A^{p+q+1}$.
Hence $\xi-\eta_r\in Z^{p, q}_{r+1}$. Since $\eta_r\in F^pA^{p+q}$ and $d\eta_r\in A^{p+r, q-r+1}\oplus A^{p+r+1, q-r}\subset F^{(p+1)+(r-1)}A^{p+q+1}$, we have $\eta_r\in Z^{p+1, q-1}_{r-1}$.
Therefore $\xi=(\xi-\eta_r)+\eta_r\in Z^{p, q}_{r+1}+Z^{p+1, q-1}_{r-1}$ which is a contradiction. Hence $d'\xi_{r-1}\notin \mbox{Im}d''$.

\item[(iii)] If $\xi_0$ is the leading term of a $d$-closed form $\tau \in F^pA^{p+q}$, then $\xi-\tau \in F^{p+1}A^{p+q}$ and $d(\xi-\tau)=d\xi\in F^{p+r}A^{p+q+1}=F^{(p+1)+(r-1)}A^{p+q+1}$.
Hence $\xi-\tau\in Z^{p+1, q-1}_{r-1}$. Then
$\xi=\tau+(\xi-\tau)\in Z^{p, q}_{\infty}+Z^{p+1,q-1}_{r-1}\subset Z^{p, q}_{r+1}+Z^{p+1,q-1}_{r-1}$ which is a contradiction.
\end{enumerate}
Hence $\xi\in \mathcal{E}^{p,q}_r$.

\end{enumerate}
\end{proof}

\begin{lemma}\label{beta}

\begin{enumerate}
\item $\mathcal{E}^{p, q-1}_0=\emptyset$ if and only if $\beta_{p, q, 0}$ is an isomorphism.

\item
For $r\geq 1$, if $\mathcal{E}^{p-r,q+r-1}_r=\emptyset$, then $\beta_{p,q,r}$ is an isomorphism.

\item
For $r\geq 1$, if $\mathcal{E}^{p-r,q+r-1}_r\neq\emptyset$, then $\beta_{p,q,j}$ is not an isomorphism for $j=1 \mbox{ or } r$.
\end{enumerate}
\end{lemma}

\begin{proof}
We note that $\beta_{p,q,r}$ is an isomorphism if and only if $B^{p,q}_{r}\subset Z^{p+1,q-1}_{r-1}+B^{p,q}_{r-1}$.
\begin{enumerate}
\item This follows from the definition.

\item
Assume that $\beta_{p,q,r}$ is not an isomorphism. Then there exists $\xi\in B^{p,q}_{r}-(Z^{p+1,q-1}_{r-1}+B^{p,q}_{r-1})$.
So $\xi=d\eta$ for some $\eta\in F^{p-r}A^{p+q-1}$.
Let
$$\eta=\eta_0+\eta_1+\cdots +\eta_k \mbox{ where } \eta_i\in A^{p-r+i,q+r-i-1}.$$
If $k\geq r$, let $\eta'=\eta_r+\cdots +\eta_k\in F^pA^{p+q-1}\subset F^{p-(r-1)}A^{p+q-1}$. Then
$d\eta'\in F^pA^{p+q}\cap d(F^{p-(r-1)}A^{p+q-1})=B^{p, q}_{r-1}$. If $d(\eta-\eta')\in Z^{p+1, q-1}_{r-1}+B^{p, q}_{r-1}$, then
$\xi=d(\eta-\eta')+d\eta'\in Z^{p+1, q-1}_{r-1}+B^{p, q}_{r-1}$ which is a contradiction. So $d(\eta-\eta')\in B^{p, q}_r-(Z^{p+1, q}_{r-1}+B^{p, q}_{r-1})$.
Hence we may assume $\xi=d\eta$ where $\eta=\eta_0+\cdots+\eta_{r-1}$.

\begin{enumerate}
\item[(i)] Comparing the degrees of $\xi$ and $d\eta$, we see that $d\eta=d'\eta_{r-1}$.

\item[(ii)] If $\eta_0=0$, then $\xi=d(\eta_1+\cdots+\eta_{r-1})\in F^pA^{p+q}\cap d(F^{p-(r-1)}A^{p+q-1})=B^{p, q}_{r-1}$ which is a contradiction. So $\eta_0\neq 0$.

\item[(iii)] If $\eta_0$ is the leading term of a $d$-closed form $\eta''$, $\eta-\eta''\in F^{p-r+1}A^{p+q-1}$ and $\xi=d\eta=d(\eta-\eta'')\in d(F^{p-(r-1)}A^{p+q-1})\cap F^pA^{p+q}=B^{p, q}_{r-1}$ which is a contradiction. Hence $\eta_0$ is not the leading term of any $d$-closed form.

\item[(iv)] If $d'\eta_{r-1}\in \Im d''$, $\xi=d\eta=d'\eta_{r-1}=-d''\eta_r$ for some $\eta_r\in A^{p,q-1}$, then $\xi =d'\eta_r-d\eta_r\in Z^{p+1, q-1}_{\infty}+B^{p,q}_{0}\subset Z^{p+1, q-1}_{r-1}+B^{p,q}_{r-1}$ which is a contradiction. Hence $d'\eta_{r-1}\notin \Im d''$.
\end{enumerate}
Therefore, $\eta\in \mathcal{E}^{p-r,q+r-1}_r$.

\item
Assume that $\mathcal{E}^{p-r,q+r-1}_r\neq \emptyset$. Let $\eta=\eta_0+\cdots +\eta_{r-1}\in \mathcal{E}^{p-r,q+r-1}_r$ where $\eta_i\in A^{p-r+i,q+r-i-1}$.
Since $d\eta\in B^{p, q}_r$, if $d\eta\notin Z^{p+1,q-1}_{r-1}+B^{p,q}_{r-1}$, $\beta_{p,q,r}$ is not an isomorphism.
So we may assume $d\eta=d'\eta_{r-1}=\xi'+d\eta'$ where $\xi'\in Z^{p+1, q-1}_{r-1}$ and $d\eta'\in B^{p,q}_{r-1}$. Let $\eta'=\eta_1'+\eta_2'+\cdots+\eta_l'$, where $\eta_i'\in A^{p-r+i,q+r-1-i}$. The degree of $d'\eta_{r-1}$ is $(p, q)$,
so by comparing degrees of both sides of $d'\eta_{r-1}=\xi'+d\eta'$, we get
$$d'\eta_{r-1}=d'\eta'_{r-1}+d''\eta'_r \mbox{ and } d''\eta'_{r-1}=0.$$
If $d'\eta'_{r-1}\in \mbox{Im}d''$,
 then $d'\eta_{r-1}\in \mbox{Im}d''$ which contradicts to the fact that $\eta\in \mathcal{E}^{p-r, q+r-1}_r$. So $d'\eta'_{r-1}\notin\Im d''$.
Note that if $\eta'_{r-1}$ is the leading term of a $d$-closed element $\tau$,
we may write $\tau=\eta'_{r-1}+\tau_r+\cdots +\tau_k$ for some $k>r-1$ and each $\tau_i\in A^{p-r+i, q+r-1-i}$. Then comparing the degrees of $d'\tau=-d''\tau$, we get
$d'\eta_{r-1}=-d''\tau_r$ which contradicts to the fact that $d'\eta_{r-1}\notin\Im d''$.

From the above verification, we see that $\eta'_{r-1}\in \mathcal{E}^{p-1, q}_1$.
Assume that $d\eta'_{r-1}\in Z^{p+1, q-1}_0+B^{p, q}_0$. Write
$d\eta'_{r-1}=\gamma+d\sigma$ where $\gamma=\gamma_1+\gamma_2+\cdots \in Z^{p+1, q-1}_0$,
$\gamma_i\in A^{p+i, q-i}$, $\sigma=\sigma_0+\sigma_1+\cdots \in B^{p, q}_0$ and
$\sigma_i\in A^{p+i, q-1-i}$. Since the degree of $d\eta'_{r-1}$ is $(p, q)$, comparing the degrees of both sides
of $d\eta'_{r-1}=\gamma+d\sigma$, we get $d\eta'_{r-1}=d''\sigma_0$ which
contradicts to the fact that $\eta'_{r-1}\in \mathcal{E}^{p-1, q}_1$. Therefore
$d\eta'_{r-1}\notin Z^{p+1, q-1}_0+B^{p, q}_0$ and hence $\beta_{p, q, 1}$ is not
an isomorphism.
\end{enumerate}
\end{proof}

\begin{theorem}\label{eqdeg}
Suppose that $(A=\oplus_{p, q\geq 0}A^{p, q}, d', d'')$ is a  double complex and $r\geq 1$. The spectral sequence $\{E^{p, q}_r\}$
induced by $A$  degenerates at $E_r$ but not at $E_{r-1}$ if and only if the following conditions hold:
\begin{enumerate}
\item
$\mathcal{E}^{p,q}_k=\emptyset$ for all $p, q\in \Z, k\geq r$ \mbox{ and }
\item
$\mathcal{E}^{p,q}_{r-1}\neq\emptyset$ for some $p, q$.
\end{enumerate}
\end{theorem}

\begin{proof}
Suppose that $\{E^{p, q}_r\}$ degenerates at $E_r$ but not at $E_{r-1}$ for some $r\geq 1$.
By Proposition \ref{alpha and beta}(2), $\alpha_{p, q, i}$ is an isomorphism for all $p, q\in \Z, i\geq r$. Then by Lemma \ref{alpha}, $\mathcal{E}^{p, q}_i=\emptyset$
for all $p, q\in \Z, i\geq r$. Since $d_{r-1}\neq 0$, by Proposition \ref{alpha and beta}(1), there are some $p, q\in \Z$ such that $\beta_{p, q, r-1}$ is not an isomorphism.
Then by Lemma \ref{beta}, $\mathcal{E}^{p-r+1, q+r-2}_{r-1}\neq \emptyset$.

Conversely, suppose that $(1)$ and $(2)$ hold. By Lemma \ref{beta}, $\beta_{p, q, k}$ is an isomorphism for all $p, q\in \Z, k\geq r$.
Then by Proposition \ref{alpha and beta}, $d_k=0$ for $k\geq r$. For the case $r=1$, by definition, $\mathcal{E}^{p, q}_0\neq \emptyset$ implies that $\beta_{p, q+1, 0}$ is not an isomorphism.
And hence by Proposition \ref{alpha and beta}, $d_0\neq 0$. For the case $r\geq 2$, if $\beta_{p, q, r-1}$ is an isomorphism for all $p, q\in \Z$, by Proposition \ref{alpha and beta}, $d_{r-1}=0$.
Then we have $d_k=0$ for $k\geq r-1$. By the proof above, $\mathcal{E}^{p, q}_k=\emptyset$ for $k\geq r-1$. In particular, $\mathcal{E}^{p, q}_{r-1}=\emptyset$ for all $p, q\in \Z$ which contradicts to our assumption (2).
Therefore there exist some $p_0, q_0$ such that $\beta_{p_0, q_0, r-1}$ is not an isomorphism. By Proposition \ref{alpha and beta},
$d_{r-1}\neq 0$.
\end{proof}

\begin{definition}\label{dd-lemma}
We say that a double complex $(A, d', d'')$ satisfies the $d'd''$-lemma at $(p,q)$ if
$$\Im d'\cap \ker d''\cap A^{p,q}=\ker d'\cap \Im d''\cap A^{p,q}=\Im d'd''\cap A^{p,q}$$
and $A$ satisfies the $d'd''$-lemma if $A$ satisfies the $d'd''$-lemma at $(p,q)$ for all $(p,q)$.
\end{definition}

Now we can give a proof of the main result Theorem \ref{main result}.

\begin{proof}
Note that by definition, $d'd''$-lemma implies that $\Im d'\cap \ker d''\cap A^{p, q}=\Im d'\cap \Im d''\cap A^{p, q}$ for all $p, q$.
Since $\{E^{p, q}_r\}$ does not degenerate at $E_0$, $\beta_{p, q, 0}$ is not an isomorphism for some $p, q$, hence by Lemma \ref{beta}, $\mathcal{E}^{p, q-1}_0\neq \emptyset$.
Assume that $\mathcal{E}^{p, q}_r\neq \emptyset$ for some $p, q\in \Z$, $r\geq 1$. Then there is $\alpha=\sum_{i=0}^{r-1}\alpha_i\in \mathcal{E}^{p,q}_r$ where $\alpha_i\in A^{p+i,q-i}$.
From the condition $d\alpha=d'\alpha_{r-1}$, we have $d''\alpha_{r-1}=-d'\alpha_{r-2}$
and hence $d''d\alpha=-d'd''\alpha_{r-1}=0$. So $d\alpha=d'\alpha_{r-1}\in (\Im d'\cap\ker d'')\cap A^{p,q}=(\Im d'\cap \Im d'')\cap A^{p,q}$.
But by the definition of $\mathcal{E}^{p, q}_r$, $d'\alpha_{r-1}\notin \Im d''$ which leads to a contradiction. Therefore by Theorem \ref{eqdeg}, $\{E^{p, q}_r\}$ degenerates at $E_1$.
\end{proof}

In the following, we apply the main result to prove the $E_1$-degeneration of a spectral sequence of bi-generalized Hermitian manifolds.
We refer the reader to \cite{G1, C} for generalized complex geometry, and to \cite{CHT} for bi-generalized complex manifolds.
We give a brief recall here. A bi-generalized complex structure on a smooth manifold $M$ is a pair $(\J_1, \J_2)$ where $\J_1, \J_2$ are commuting generalized complex structures on $M$. A
bi-generalized complex manifold is a smooth manifold $M$ with a bi-generalized complex structure. A bi-generalized Hermitian manifold $(M, \J_1, \J_2, \G)$
is an oriented bi-generalized complex manifold $(M, \J_1, \J_2)$ with a generalized metric $\G$ which commutes with $\J_1$ and $\J_2$.
We define
$$U^{p, q}:=U^p_1\cap U^q_2$$ where $U^p_1, U^q_2\subset \Gamma(\Lambda^*\TM\otimes \C)$ are eigenspaces of $\J_1, \J_2$ associated to the eigenvalues $ip$ and $iq$ respectively
and $\TM=TM\oplus T^*M$ is the generalized tangent space. It can
be shown that the exterior derivative $d$ is an operator from
$U^{p, q}$ to $U^{p+1, q+1}\oplus U^{p+1, q-1}\oplus U^{p-1, q+1} \oplus U^{p-1, q-1}$ and we write
$$\delta_+:U^{p, q} \rightarrow U^{p+1, q+1}, \delta_-:U^{p, q} \rightarrow U^{p+1, q-1}$$
for the projection of $d$ into corresponding spaces.

\begin{definition}
On a bi-generalized Hermitian manifold $M$, there is a double complex $\{(A, d', d'')\}$ given by
$$A^{p, q}:=U^{p+q, p-q}, d'=\delta_+, d''=\delta_-$$
We call the spectral sequence $\{E^{*, *}_*\}$ associated to this double complex the $\partial_1$-Hodge-de Rham spectral sequence.
\end{definition}

By Theorem \ref{main result}, we have the following result.

\begin{theorem}
Suppose that $M$ is a compact bi-generalized Hermitian manifold which satisfies the $\delta_+\delta_-$-lemma and has positive dimension. Then the $\partial_1$-Hodge-de Rham spectral sequence degenerates at $E_1$.
\end{theorem}

Now we give a proof of the $E_1$-degeneration of the $\partial_1$-Hodge-de Rham spectral sequence.
\begin{proof}
Since $\bigoplus_{p, q}U^{p, q}=\Omega^{\bullet}(M)\otimes \C$ (see \cite{C2}, pg 36) where $\Omega^{\bullet}(M)$ is the collection of smooth forms on $M$, some $U^{p, q}$ is not empty. The space $U^{p, q}$ is a $C^{\infty}(M, \C)$-module
where $C^{\infty}(M, \C)$ is the ring of complex-valued smooth functions on $M$, and $M$ has positive dimension, therefore $U^{p, q}$ is an infinite dimensional complex vector space . If $\delta_-$ is a zero map, we have
$H^{p, q}_{\delta_-}(M)=U^{p, q}$ for all $p, q$. But $M$ is compact, this contradicts to the fact that $H^{p, q}_{\delta_-}(M)$ is finite dimensional(\cite[Theorem 2.14, Corollary 3.11]{CHT}). Hence $\delta_-$ is not the zero map.
and the spectral sequence does not degenerate at $E_0$.
Since we assume that $M$ satisfies the $\delta_+\delta_-$-lemma, by Theorem \ref{main result}, the spectral sequence degenerates at $E_1$.
\end{proof}

\bibliographystyle{alpha}

\end{document}